\documentclass[12pt,reqno]{amsart}
\usepackage{amsmath}
\usepackage{amssymb}
\usepackage{amstext}
\usepackage{a4wide}
\usepackage{graphicx}
\usepackage{bm}
\allowdisplaybreaks \numberwithin{equation}{section}
\usepackage{color}
\usepackage{cases}

\numberwithin{equation}{section}

\newtheorem{theorem}{Theorem}[section]

\newtheorem{lemma}[theorem]{Lemma}

\theoremstyle{definition}

\theoremstyle{remark}
\newtheorem{remark}[theorem]{Remark}
\newtheorem{remarks}[theorem]{Remarks}

\begin{document}

\title
{Traveling vortex pairs for 2D incompressible Euler equations}

 \author{Daomin Cao, Shanfa Lai,  Weicheng Zhan}

\address{Institute of Applied Mathematics, Chinese Academy of Sciences, Beijing 100190, and University of Chinese Academy of Sciences, Beijing 100049,  P.R. China}
\email{dmcao@amt.ac.cn}
\address{Institute of Applied Mathematics, Chinese Academy of Sciences, Beijing 100190, and University of Chinese Academy of Sciences, Beijing 100049,  P.R. China}
\email{laishanfa19@mails.ucas.edu.cn}
\address{Institute of Applied Mathematics, Chinese Academy of Sciences, Beijing 100190, and University of Chinese Academy of Sciences, Beijing 100049,  P.R. China}
\email{zhanweicheng16@mails.ucas.ac.cn}


\begin{abstract}
 In this paper, we study desingularization of vortices for the two-dimensional incompressible Euler equations in the full plane. We construct a family of steady vortex pairs for the Euler equations with a general vorticity function, which constitutes a desingularization of a pair of point vortices with equal magnitude and opposite signs. The results are obtained by using an improved vorticity method.
\end{abstract}

\maketitle

\section{Introduction and Main results}
The incompressible planar flow without external force is governed by the following Euler equations
\begin{numcases}
{}
\label{0-1} \partial_t \mathbf{v}+(\mathbf{v}\cdot\nabla)\mathbf{v}=-\nabla P &\\
\label{0-2}  \nabla\cdot\mathbf{v}=0,  &
\end{numcases}
where $\mathbf{v}=(v_1,v_2)$ is the velocity field, $P$ is the scalar pressure.

Let $\omega=\nabla\times \mathbf{v}=\partial_1v_2-\partial_2v_1$ be the corresponding vorticity of the flow. Then the vorticity $\omega$ obeys the following  transport equation
\begin{equation}\label{0-3}
  \partial_t\omega+\mathbf{v}\cdot \nabla\omega=0.
\end{equation}
On the other hand, by virtue of equation \eqref{0-2}, there exists a Stokes stream function $\Psi(x,t)$ such that
\begin{equation*}
  \mathbf{v}=\nabla^\perp \Psi\equiv (\partial_{x_2} \Psi, -\partial_{x_1}\Psi).
\end{equation*}
Note that by definition we have $-\Delta \Psi =\omega$. Thus the velocity field $\mathbf{v}$ can be recovered from the vorticity function $\omega$ through the Biot-Savart law
\begin{equation*}
  \mathbf{v}=\nabla^\perp (-\Delta)^{-1}\omega=-\frac{1}{2\pi} \frac{x^\perp}{|x|^2}\ast \omega.
\end{equation*}
Based on this transformation, we can rewrite the Euler equations as the vorticity-stream formulation (see \cite{MB})
\begin{numcases}
{}
\label{1-1} \partial_t\omega+\mathbf{v}\cdot \nabla\omega=0, &\\
\label{1-2}  \mathbf{v}=\nabla^\perp \Psi,\ \Psi=(-\Delta)^{-1}\omega.  &
\end{numcases}

One of the most spectacular phenomena in nature is a fluid with concentrated vortices, exemplified by atmospheric cyclones, whirlwinds and tornados, oceanic vortices, and whirlpools on a water surface. There flows are usually characterized by the localization of vorticity a finite number of small regions, beyond which the vorticity is either absent or rapidly falls down to zero. Mathematically, one can consider singular solutions of the Euler equations given by
\begin{equation}\label{po1}
  \omega=\sum_{i=1}^{N}\kappa_i \bm{\delta}\left(x-\Im_i(t)\right),
\end{equation}
corresponding to
\begin{equation*}
  \mathbf{v}(x)=-\sum_{i=1}^{N}\frac{\kappa_i}{2\pi}\frac{(x-\Im_i(t))^\perp}{|x-\Im_i(t)|^2},
\end{equation*}
and the position of the vortices $\Im_i:\mathbb{R}\to \mathbb{R}^2$ satisfy
\begin{equation}\label{po3}
  \frac{d}{dt}{\Im}_i(t)=-\sum_{j=1, j\neq i}^{N}\frac{\kappa_j}{2\pi}\frac{(\Im_i(t)-\Im_j(t))^\perp}{|\Im_i(t)-\Im_j(t)|^2}.
\end{equation}
Here $\bm{\delta}(x)$ denotes the standard Dirac mass at the origin, and the constants $\kappa_i$ are called the intensities of the vortices $\Im_i(t)$. The Hamiltonian system \eqref{po3} is now generally referred to as the point vortex model or the Kirchhoff-Routh model (see \cite{Mar}).

We are concerned in this paper with the solutions with sharply concentrated vorticities around the points $\Im_1(t),...,\Im_N(t)$ which are idealized as regular solutions that approximate the point vortex model. Finding such regular solutions is classical vortex desingularization problem, see \cite{Cao1, Cao2, Da1, Mar, SV} and the references therein. Note that system \eqref{po3} admits a simple traveling solution which is of the form
\begin{equation}\label{po4}
  {\Im}_i(t)=b_i-Wt \mathbf{e}_2,\ \  i=1,..., N,
\end{equation}
where $b_1,..., b_N$ are points in $\mathbb{R}^2$, the constant $W>0$ is the speed, and $\mathbf{e}_2=(0,1)$. In this paper, we focus on traveling solutions of the form \eqref{po4}. For simplicity, we consider the traveling vortex pair, namely the solution with $N=2$, and
\begin{equation}\label{po5}
     b_1  =-b_2=r\mathbf{e}_1,\ \ \mathbf{e}_1=(1,0),\ \ \kappa_1=-\kappa_2=\kappa,\ \ W =\frac{\kappa}{4\pi r},
\end{equation}
where $r>0$ and $\kappa>0$.

The study of planar vortex pairs is one of the main objects of mathematical fluid mechanics, which has been carried out for more than a century. They are one instance of a large collection of coherent structures found in two dimensional vortex dynamics. The literature of vortex pairs can be traced back to the work of Pocklington \cite{Poc} in 1895. Vortex pairs are theoretical models of coherent vortex structures in large-scale geophysical flows, see \cite{Flo, Hei}.

To study vortex desingularization problem of the traveling vortex pair, we can look for traveling solutions to \eqref{1-1}-\eqref{1-2} by requiring that
\begin{equation*}
  \omega(x_1,x_2, t)=\zeta(x_1, x_2+Wt),
\end{equation*}
for some profile function $\zeta$ defined on $\mathbb{R}^2$. In this case, $\Psi(x_1,x_2, t)=\phi(x_1, x_2+Wt)$ with $\phi=(-\Delta)^{-1}\zeta$. Therefore the vorticity-stream formulation \eqref{1-1}-\eqref{1-2} is reduced to a stationary problem
\begin{equation}\label{po7}
  \nabla^\perp (\phi-Wx_1)\cdot \nabla \zeta=0,\ \ \phi=(-\Delta)^{-1}\zeta.
\end{equation}
The condition that $\omega$ approximates \eqref{po1} now becomes
\begin{equation}\label{po8}
  \zeta(x)\approx\sum_{i=1}^{2}\kappa_i \bm{\delta}\left(x-b_i\right).
\end{equation}
We consider flows are symmetric about the $x_2$-axis, and restrict attention henceforth to the half-plane $\Pi=\{x\in \mathbb{R}^2\mid x_1>0\}$. Specifically, we assume
\begin{equation*}
  \zeta(x_1, x_2)=-\zeta(-x_1, x_2),
\end{equation*}
so that $\phi(x_1, x_2)=-\phi(-x_1, x_2)$. A natural way of obtaining solutions to the stationary problem \eqref{po7} is to impose that $\zeta(x)$ and $\phi(x)-Wx_1$ are locally functional dependent. On the other hand, to get \eqref{po8}, we can introduce a parameter and study the asymptotic behavior of the solutions with respect to this parameter. More precisely, we consider the following equation
\begin{equation}\label{fr}
  -\Delta \phi=\left\{
   \begin{array}{ll}
     \lambda f(\phi-Wx_1-\mu) &\text{in}\ \  A, \\
    \ \ \ \ \ \ \  0 &\text{in}\ \  \Pi-\overline{A},
     \end{array}
     \right.
\end{equation}
where $\lambda>0$ is a vortex strength parameter, $f:\mathbb{R}\to \mathbb{R}$ is a vorticity function, $\mu\ge0$ is flux constant, and $A=\{x\in \Pi\mid \zeta(x)\neq 0\}$ is half the vortex pair (a priori unknown). Here $\nabla \phi$ should be continuous across $\partial A$, the boundary of $A$. Since the flow is symmetric about the $x_2$-axis, the $x_2$-axis must be a streamline. We set $\phi=0$ on the $x_2$-axis. Moreover, we impose uniform flow at infinity by means of $\nabla \phi \to 0$ as $|x|\to \infty$. The unknown boundary $\partial A$ is also a streamline and we set $\phi=0$ on $\partial A$. We can remove this consideration by requiring that
\begin{itemize}
\item [($\text{H1}$)]\ \ $f(s)=0\ \ \text{for}\ \ s\le 0;\ \ f(s)>0\ \ \text{for}\ \ s>0.$
\end{itemize}
Then we may restate \eqref{fr} as
\begin{equation}\label{fr2}
  -\Delta \phi= \lambda f(\phi-Wx_1-\mu)\ \ \text{in}\ \Pi,
\end{equation}
where $A=\{x\in \Pi\mid \phi(x)-Wx_1-\mu>0 \}$, since the maximum principle implies that
\begin{equation*}
  \phi(x)-Wx_1-\mu>0 \ \text{in}\ A\  \ \ \text{and}\ \ \ \phi(x)-Wx_1-\mu<0\ \text{in}\ \Pi-\overline{A}.
\end{equation*}

In 1906, Lamb \cite{Lamb} noted an explicit solution to \eqref{fr2} with $f(s)=s_+$ and $\mu=0$, where $s_+:=\max\{s,0\}$. This explicit solution is now generally referred to as the Lamb dipole or Chaplygin-Lamb dipole, see \cite{Mel}. Its vorticity $\zeta$ is positive inside a semicircular region, that outside this region the flow is irrotational. In conjunction with its reflection in the $x_2$-axis, this flow constitutes a circular vortex. The uniqueness for the circular vortex pair was established by Burton \cite{Bu1, Bu2}. Recently, Abe and Choi \cite{Abe} considered orbital stability of the Lamb dipole.

Exact solutions of \eqref{fr} are known only in special cases. Beside those exact solutions, the existence and abundance of steady vortex pairs has been rigorously established. Norbury \cite{Nor} constructed a wide class of variational solutions based on the variational approach proposed in \cite{Fra}. In \cite{Nor}, the vorticity function $f$ is prescribed, and the vortex strength parameter $\lambda$ arises as a Lagrange multiplier and hence is left undetermined. Turkington \cite{Tur83} studied the case when $f$ is the Heaviside function and constructed a family of desingularization solutions for \eqref{po8}. In \cite{Bad, Bu0}, Burton and Badiani proved the existence of steady vortex pairs with a prescribed distribution of the vorticity. In this setting, the vorticity function $f$ is a nondecreasing function but unknown a priori. Ambrosetti and Yang \cite{Amb, Yang} studied the existence of solutions by using the mountain pass lemma. They also studied the asymptotic behavior of the solutions when $\lambda \to +\infty$. However, their limiting objects are degenerate vortex pairs with vanishing vorticity and hence do not answer the question about vortex desingularization (cf.~\cite{SV}). Recently, Smets and Van Schaftingen \cite{SV} obtained the desingularization of vortex pairs with $f(s)=s_+^p$ for $p>1$. For numerical studies, see, e.g., \cite{Del, Kiz}.

The purpose of this paper is to study vortex desingularization problem of the traveling vortex pair with a general vorticity function $f$. As mentioned above, for some special nonlinearities $f$, there are already some desingularization results. Our goal here is to generalize these results to a more general nonlinearity and hence enrich the known solutions of steady vortex pairs.

Our first main result is as follows.
\begin{theorem}\label{thm1}
  Consider the traveling vortex pair given by \eqref{po5}. Suppose $f$ is a bounded nondecreasing function satisfying (H1). Then for $\lambda>0$ large there is a travelling solution $(\omega^\lambda,\Psi^\lambda)$ to \eqref{1-1}-\eqref{1-2} given by
  \begin{equation*}
    \omega^\lambda(x_1,x_2, t)=\zeta^\lambda(x_1, x_2+Wt)=\lambda f\left(\Psi^\lambda(x_1, x_2+Wt)-Wx_1-\mu^\lambda\right),
  \end{equation*}
  for all $(x_1,x_2)\in \Pi$, and which satisfies
  \begin{equation*}
    \zeta^\lambda(x)\rightharpoonup \kappa \bm{\delta}\left(x-b_1\right)-\kappa \bm{\delta}\left(x-b_2\right)\ \text{as}\ \lambda \to +\infty,\ \ \text{supp} (\zeta^\lambda) \subset  \bigcup^2_{i=1}B_{L\lambda^{-\frac{1}{2}}}(b_i),
 \end{equation*}
 where the convergence is in the sense of measures and $L>0$ is a constant independent of $\lambda$. Moreover, one has
  \begin{equation*}
    \zeta(x_1,x_2)=-\zeta(-x_1,x_2)=\zeta(x_1,-x_2),
  \end{equation*}
  for all $(x_1,x_2)\in \mathbb{R}^2$, and the corresponding flux constant $\mu^\lambda$ satisfies for $\lambda$ large
 \begin{equation*}
   \mu^\lambda=\frac{\kappa}{4\pi}\log \lambda+O(1).
 \end{equation*}
\end{theorem}

\begin{remarks}
 We give some remarks about the above theorem.
  \begin{itemize}
    \item [1)] Notice that $f$ is allowed to be discontinuous. Nevertheless, the monotonicity of $f$ is enough to ensure the above solutions do give rise to traveling solutions of the Euler equations \eqref{0-1}-\eqref{0-2}. We refer the reader to Section 5 in \cite{Bu0} for a demonstration. Indeed, these solutions are classical solutions when $f$ is smooth enough. This can be achieved by the classical elliptic estimates.
    \item [2)] When $f$ is the Heaviside function, we reobtain Turkington's result in \cite{Tur83}. Our first result can be viewed as a generalization of Turkington's work.
  \end{itemize}
\end{remarks}

Our second result is concerned with the case when $f$ is unbounded. For technical reasons, we make some assumptions on $f$.
\begin{itemize}
\item[(H2)] There exist $\vartheta_0\in(0,1)$ and $\vartheta_1>0$ such that
\[F(s):=\int_0^sf(t)dt\leq \vartheta_0f(s)s+\vartheta_1f(s),\,\,\,\forall\,s\geq0.\]
\item[(H3)] There holds
\[\liminf_{s\to+\infty} f(s)e^{-\vartheta_2s}=0,\ \ \ \text{where}\ \  \vartheta_2=\frac{4\pi\min\{2\vartheta_0, 2-2\vartheta_0\} }{\kappa}.\]
\end{itemize}
Assumption (H2) is a Ambrosetti-Rabinowitz-type condition, cf.~condition $(p5)$ in \cite{Am}. This assumption implies that $f$ is unbounded (see \cite{Ni}). Assumption (H3) requires that the function $f$ does not grow too fast. Many profile functions that frequently appear in nonlinear elliptic equations satisfy (H1)-(H3), for example, $f(s)=s_+^p$ with $p\in (0,+\infty)$. Under these assumptions, we have
\begin{theorem}\label{thm2}
  Suppose that $f$ is a nondecreasing function satisfying (H1)-(H3), then the conclusion of Theorem \ref{thm1} still holds.
\end{theorem}

\begin{remarks}
When $f(s)=s_+^p$ for $p>1$, Smets and Van Schaftingen have obtained desingularization result in \cite{SV}.
\end{remarks}

Roughly speaking, there are two methods to study this problem, namely the stream-function method and the vorticity method. The stream-function method is to find a $\psi$ satisfying \eqref{fr2} with the desired properties; see, e.g., \cite{Amb, Nor, SV, Yang}. The vorticity method focuses on the vorticity of the flow; see, e.g., \cite{Bad, Bu0, Tur83}.  The key idea is to solve a variational problem for the vorticity $\zeta$. Compare to the stream-function method, the vorticity method has a strong physical motivation. By equation \eqref{0-3} the vorticity is conserved along particle trajectories for the flow. Moreover, the fluid impulse is also conserved for all time (see \cite{Mar, MB}). These facts can be used to establish the stability of the solution from the viewpoint of vorticity. We refer the reader to \cite{Abe, Bu3, Bu4, Bu5} for some results in this direction. For this reason, we prefer to construct solutions by the vorticity method. Mathematically, this method can be regarded as a dual variational principle; see \cite{Am, Ber, Bu2, ET, St, Tol} for example. The proofs of Theorems \ref{thm1} and \ref{thm2} are provided in the next section.

\section{Proofs of Theorems \ref{thm1} and \ref{thm2}}
In this section we will give proof for Theorems \ref{thm1} and \ref{thm2}. Let $G(\cdot,\cdot)$ be the Green's function for $-\Delta$ in $\Pi$ with zero Dirichlet data, namely,
\begin{equation*}
  G(x,y)=\frac{1}{2\pi}\log\frac{|\bar{x}-y|}{|x-y|},
\end{equation*}
where $\bar{x}=(-x_1,x_2)$ denotes reflection of $x=(x_1,x_2)$ in the $x_2$-axis. Define the Green's operator $\mathcal{G}$ as follows
\begin{equation*}
  \mathcal{G}\zeta(x)=\int_\Pi G(x,y)\zeta(y)dy.
\end{equation*}

\subsection{Proof of Theorem \ref{thm1}}
In this subsection, we consider the bounded case. For the sake of clarity, we will split the proof into several lemmas.

Let $J$ be the conjugate function to $F$ defined by $J(s)=\sup_{t\in \mathbb{R}}\left(st-F(t)\right)$. We shall use $\partial J(s)$ to denote the subgradient of $J$ at $s$ (see \cite{Roc}).

\subsubsection{Variational problem}
Let $D=\{x\in \Pi~|~r/2<x_1<2r,\ -1<x_2<1\}$ and
\begin{equation*}
\mathcal{A}=\{\zeta\in L^\infty(\Pi)~|~  \zeta \ge 0~ \text{a.e.}, \int_{\Pi}\zeta dx \le \kappa,~ supp(\zeta)\subseteq D \},
\end{equation*}
For $\varepsilon>0$ we define $\mathcal{E}_\varepsilon(\zeta)\in [-\infty,+\infty)$ by
\begin{equation*}
  \mathcal{E}_\varepsilon(\zeta)=\frac{1}{2}\int_D{\zeta \mathcal{G}\zeta}dx-W\int_{D}x_1\zeta dx-\frac{1}{\varepsilon^2}\int_D J(\varepsilon^2\zeta)dx.
\end{equation*}
 We assume that
\begin{equation}\label{2-0}
  \frac{\sup_{\mathbb{R}^2}f}{\varepsilon^2} |D|>\kappa,
\end{equation}
 where $|\cdot|$ denotes the two-dimensional Lebesgue measure.

Note that $\mathcal{E}_\varepsilon\not \equiv -\infty$ on $\mathcal{A}$. We will seek maximizers of $\mathcal{E}_\varepsilon$ relative to $\mathcal{A}$. Let ${\zeta}^*$ be the Steiner symmetrization of $\zeta$ with respect to the line $x_2=0$ in $\Pi$ (see \cite{Bu0, Nor}).

We first have
\begin{lemma}\label{le1}
  There exists $\zeta^{\varepsilon}=(\zeta^\varepsilon)^* \in \mathcal{A}$ such that
\begin{equation*}
 \mathcal{E}_\varepsilon(\zeta^{\varepsilon})= \max_{\tilde{\zeta} \in \mathcal{A}}\mathcal{E}_\varepsilon(\tilde{\zeta})<+\infty.
\end{equation*}
Moreover, one has
\begin{equation}\label{2-1}
  \varepsilon^2\zeta^\varepsilon(x)\le \sup_{x\in \mathbb{R}} f,\ \ \ \forall~x\in \Pi.
\end{equation}
\end{lemma}

\begin{proof}
  Notice that the effect domain of $J(\cdot)$ is contained in $(-\infty, \sup f]$. We may take a sequence $\{\zeta_{k}\}\subset \mathcal{A}$ such that as $j\to +\infty$
\begin{equation*}
  \begin{split}
        \mathcal{E}_\varepsilon(\zeta_{k}) & \to \sup\{\mathcal{E}_\varepsilon(\tilde{\zeta})~|~\tilde{\zeta}\in \mathcal{A}\}, \\
        \zeta_{k} & \to \zeta\in L^{\infty}({D})~~\text{weakly-star}.
  \end{split}
\end{equation*}
Clearly one has $\zeta\in \mathcal{A}$. Using the standard arguments (see \cite{Bu0, Nor}), we may assume that $\zeta_{k}=(\zeta_{k})^*$, and hence $\zeta=\zeta^*$. Since $G(\cdot,\cdot)\in L^1(D\times D)$, we have
\begin{equation*}
      \lim_{k\to +\infty}\int_D{\zeta_{k} \mathcal{G}\zeta_{k}}dx = \int_D{\zeta \mathcal{G}\zeta}dx,\ \text{as}\ k\to +\infty.
\end{equation*}
On the other hand, we have the lower semicontinuity of the rest of terms, namely,
\begin{equation*}
  \begin{split}
    \liminf_{k\to +\infty} \int_{D}x_1\zeta_k dx  & \ge  \int_{\Pi}x_1\zeta dx, \\
    \liminf_{k\to +\infty}\int_D J(\varepsilon^2\zeta_k)dx   & \ge \int_D J(\varepsilon^2\zeta)dx.
  \end{split}
\end{equation*}
Consequently, we conclude that $ \mathcal{E}_\varepsilon(\zeta)=\lim_{k\to +\infty} \mathcal{E}_\varepsilon(\zeta_k)=\sup_{\mathcal{A}} \mathcal{E}_\varepsilon$, with $\zeta\in \mathcal{A}$, which completes the proof.
\end{proof}

Since $\zeta^\varepsilon \in L^\infty(\Pi)$, it follows that $\mathcal{G}\zeta^\varepsilon\in C^1_{\text{loc}}(\Pi)$. Moreover, by $\zeta^{\varepsilon}=(\zeta^\varepsilon)^*$ we can conclude that $\mathcal{G}\zeta^\varepsilon$ is symmetric decreasing in $x_2$ and $\partial_{x_2}\mathcal{G}\zeta^\varepsilon(x) <0$ if $\zeta^\varepsilon\not\equiv 0$ and $x_2>0$. Therefore every level set of $\mathcal{G}\zeta^{\varepsilon}-{Wx_1}$ has measure zero by the implicit function theorem.

\begin{lemma}\label{le2}
Let $\zeta^{\varepsilon}$ be a maximizer as in Lemma \ref{le1}, then there exists a Lagrange multiplier $\mu^{\varepsilon}\ge 0$ such that
\begin{equation}\label{2-2}
\zeta^{\varepsilon}=\frac{1}{ \varepsilon^2}f\left(\psi^{\varepsilon}\right) \ \ a.e.\  \text{in}\  D,
\end{equation}
where
\begin{equation*}
 \psi^{\varepsilon}=\mathcal{G}\zeta^{\varepsilon}-Wx_1-\mu^{\varepsilon}.
\end{equation*}
Moreover, it holds $\int_D\zeta^\varepsilon dx=\kappa$ provided $\zeta^\varepsilon\not\equiv0$ and every $\mu^\varepsilon$ satisfying \eqref{2-2} is positive.
\end{lemma}

\begin{proof}
  We may assume that $\zeta^\varepsilon\not\equiv0$, otherwise the assertion is trivial. Let the function $\tau: \mathbb{R}\to [0,+\infty)$ be
 \begin{equation*}
   \tau(t)=\frac{1}{\varepsilon^2}\int_D f\left(\mathcal{G}\zeta^{\varepsilon}-Wx_1+t\right)dx.
 \end{equation*}
 It is clear that $\tau(t)$ is monotone increasing and $\lim_{t\to -\infty}\tau(t)=0$, $\lim_{t\to +\infty}\tau(t)>\kappa$ by assumption \eqref{2-0}. Recall that every level set of $\mathcal{G}\zeta^\varepsilon-Wx_1$ has measure zero and $f$ has at most countable discontinuities. Using these facts, it is easy to verify that $\tau(t)$ is a continuous function. So there exists a $\mu^\varepsilon\in \mathbb{R}$ such that $\tau(-\mu^\varepsilon)=\int_D \zeta^\varepsilon dx$. Let
 \begin{equation*}
    \tilde{\zeta}^\varepsilon=\frac{1}{\varepsilon^2}f\left( \mathcal{G}\zeta^{\varepsilon}-Wx_1-\mu^\varepsilon\right)\chi_{_{D}}.
 \end{equation*}
Then $\tilde{\zeta}^\varepsilon\in \mathcal{A}$. We now claim that $\zeta^\varepsilon= \tilde{\zeta}^\varepsilon$. Indeed, let $\zeta=(\zeta^\varepsilon+\tilde{\zeta}^\varepsilon)/2 \in \mathcal{A}$. Using the convexity of $J(\cdot)$, we derive from $\mathcal{E}_\varepsilon(\zeta^\varepsilon
)\ge \mathcal{E}_\varepsilon(\zeta)$:
\begin{equation}\label{2-4}
\begin{split}
     \frac{1}{\varepsilon^2}\int_D J(\varepsilon^2\tilde{\zeta}^\varepsilon)dx&- \frac{1}{\varepsilon^2}\int_D J(\varepsilon^2{\zeta}^\varepsilon)dx\ge \\
     & \int_D\left(\mathcal{G}\zeta^{\varepsilon}-Wx_1\right)\left(\tilde{\zeta}^\varepsilon-{\zeta}^\varepsilon\right)dx+\frac{1}{4}\int_D \left(\tilde{\zeta}^\varepsilon-{\zeta}^\varepsilon\right)\mathcal{G}\left(\tilde{\zeta}^\varepsilon-{\zeta}^\varepsilon\right)dx.
\end{split}
\end{equation}
 Since $\partial J(\cdot)$ and $f(\cdot)$ are inverse graphs, it follows that
\begin{equation*}
  \mathcal{G}\zeta^{\varepsilon}-Wx_1-\mu^\varepsilon\in \partial J(\tilde{\zeta}^\varepsilon),\ \ \text{in}\ D.
\end{equation*}
By the convexity of $J(\cdot)$, we then have
\begin{equation}\label{2-5}
   \frac{1}{\varepsilon^2}\int_D J(\varepsilon^2\tilde{\zeta}^\varepsilon)dx- \frac{1}{\varepsilon^2}\int_D J(\varepsilon^2{\zeta}^\varepsilon)dx\le \int_D \left(  \mathcal{G}\zeta^{\varepsilon}-Wx_1-\mu^\varepsilon  \right)\left(\tilde{\zeta}^\varepsilon-{\zeta}^\varepsilon\right)dx.
\end{equation}
 Observing that $\int_D (\tilde{\zeta}^\varepsilon-{\zeta}^\varepsilon)dx=0$, we derive from \eqref{2-4} and \eqref{2-5} that
 \begin{equation*}
   \int_D \left(\tilde{\zeta}^\varepsilon-{\zeta}^\varepsilon\right)\mathcal{G}\left(\tilde{\zeta}^\varepsilon-{\zeta}^\varepsilon\right)dx\le 0
 \end{equation*}
 whence $\zeta^\varepsilon= \tilde{\zeta}^\varepsilon$. We further show that we can replace $\mu^\varepsilon$ with $\max\{\mu^\varepsilon, 0\}$. Suppose now $\mu^\varepsilon<0$, then
 \begin{equation*}
    \bar{\zeta}^\varepsilon=\frac{1}{\varepsilon^2}f\left(\mathcal{G}\zeta^{\varepsilon}-Wx_1\right)\chi_{_{D}}\in \mathcal{A}.
 \end{equation*}
As argued above, we now have
  \begin{equation*}
   \int_D \left(\bar{\zeta}^\varepsilon-{\zeta}^\varepsilon\right)\mathcal{G}\left(\bar{\zeta}^\varepsilon-{\zeta}^\varepsilon\right)d\nu\le 0,
 \end{equation*}
 which implies $\zeta^\varepsilon= \bar{\zeta}^\varepsilon$.

It remains to show that $\int_D\zeta^\varepsilon d\nu =\kappa$ if $\zeta^\varepsilon\not\equiv0$ and every $\mu^\varepsilon$ satisfying \eqref{2-2} is positive. We argue by contradiction. Suppose $\int_D\zeta^\varepsilon d\nu<\kappa$, then we must have $\tau(0)>\kappa$. Indeed, if $\tau(0)=\kappa$, then $\zeta^\varepsilon= \bar{\zeta}^\varepsilon$ as explained above. This is a contradiction. So
we can find a number $\hat{\mu}^\varepsilon>0$, such that
\begin{equation*}
  \hat{\zeta}^\varepsilon:=\frac{1}{\varepsilon^2}f\left( \mathcal{G}\zeta^{\varepsilon}-Wx_1-\hat{\mu}^\varepsilon\right)\chi_{_{D}}\in \mathcal{A}\ \ \text{and}\ \int_D \hat{\zeta}^\varepsilon d\nu>\int_D{\zeta}^\varepsilon d\nu.
\end{equation*}
Arguing as before, we can now obtain
  \begin{equation*}
   \int_D \left(\hat{\zeta}^\varepsilon-{\zeta}^\varepsilon\right)\mathcal{G}\left(\hat{\zeta}^\varepsilon-{\zeta}^\varepsilon\right)d\nu< 0.
 \end{equation*}
  This leads to a contradiction and the proof is thus complete.
\end{proof}

\subsubsection{Asymptotic behavior}
In the following, we study the asymptotic behavior of $\zeta^{\varepsilon}$ when $\varepsilon \to 0^+$. In the sequel we shall denote $C$ for positive constant independent of $\varepsilon$.

To begin with, we give a lower bound of the energy.
\begin{lemma}\label{le3}
 There exists $C>0$ such that
\begin{equation*}
  \mathcal{E}_\varepsilon(\zeta^{\varepsilon})\ge \frac{\kappa^2}{4\pi}\log{\frac{1}{\varepsilon}}-C.
\end{equation*}
\end{lemma}

\begin{proof}
  The key idea is to select a suitable test function. Let $x_0=(a,0)\in D$ and $$\tilde{\zeta}^\varepsilon=\frac{\sup f}{2\varepsilon^2} \chi_{_{B_{\varepsilon\sqrt{{2\kappa}/{\pi\sup f}}}(x_0)}},$$
where $\chi_{_A}$ denotes the characteristic function of a set $A$.
It is clear that $\tilde{\zeta}^{\varepsilon} \in \mathcal{A}$ if $\varepsilon$ is sufficiently small. Note that the effective domain of $J$ is contained in $[0,\sup f]$, hence $J$ is bounded on $[0,\sup f/2]$. By a simple calculation, we get
\begin{equation*}
  \mathcal{E}(\tilde{\zeta}^{\varepsilon})\ge \frac{\kappa^2}{4\pi}\ln{\frac{1}{\varepsilon}}-C.
\end{equation*}
Since $\zeta^\varepsilon$ is a maximizer, we have $\mathcal{E}(\zeta^\varepsilon)\ge \mathcal{E}(\tilde{\zeta}^{\varepsilon})$ and the proof is thus complete.
\end{proof}

Let us introduce the energy of the vortex core as follows
\begin{equation*}
  \mathcal{I}_\varepsilon=\int_D \zeta^\varepsilon \psi^\varepsilon dx.
\end{equation*}

\begin{lemma}\label{le4}
 $\mathcal{I}_\varepsilon\le C$.
\end{lemma}
\begin{proof}
If we take $\psi_+^\varepsilon\in H^1_0(\Pi)$ as a test function, we obtain
  \begin{equation}\label{2-6}
  \int_\Pi {|\nabla \psi^\varepsilon_+|^2} dx =\int_D \zeta^\varepsilon \psi^\varepsilon dx,
\end{equation}
Set $D_1^\varepsilon=\{x\in D \mid \psi^\varepsilon(x)>1\}$. Then $\zeta^\varepsilon(x)\ge f(1)/\varepsilon^2$ a.e. on $D_1^\varepsilon$ and $|D_1^\varepsilon|\le C\varepsilon^2$. By H\"older's inequality and Sobolev's inequality, we have
\begin{equation}\label{2-7}
\begin{split}
   \int_D \zeta^\varepsilon \psi^\varepsilon dx & \le \int_D \zeta^\varepsilon(\psi^\varepsilon-1)_+dx+\kappa \\
     & \le C\varepsilon^{-2}|D^\varepsilon_1|^\frac{1}{2} \left(\int_D (\psi^\varepsilon-1)_+^2dx\right)^\frac{1}{2}+\kappa \\
     & \le C\varepsilon^{-2}|D^\varepsilon_1|^\frac{1}{2} \left(\int_D \nabla (\psi^\varepsilon-1)_+ dx +\int_D (\psi^\varepsilon-1)_+ dx \right)+\kappa   \\
     & \le C \left(\int_D |\nabla \psi_+|^2 dx\right)^\frac{1}{2}+C\varepsilon \int_D \zeta^\varepsilon \psi^\varepsilon dx+C,
\end{split}
\end{equation}
Combining \eqref{2-6} and \eqref{2-7}, we conclude that $\mathcal{I}_\varepsilon \le C$.
\end{proof}
Let $\mathcal{J_\varepsilon}=\varepsilon^{-2}\int_D J(\varepsilon^2 \zeta^\varepsilon)dx$. By convexity we have (see \cite{Roc})
\begin{equation*}
   J(\varepsilon^2\zeta^\varepsilon)+ F(\psi^\varepsilon) = \varepsilon^2 \zeta^\varepsilon \psi^\varepsilon, \ \ \text{a.e.}\ \text{on}\ D,
\end{equation*}
Hence $\mathcal{J_\varepsilon} \le 2\mathcal{I}_\varepsilon$. In conclusion, we get
\begin{lemma}\label{le5}
  $\mathcal{J_\varepsilon}\le C$.
\end{lemma}

Now we turn to estimate the Lagrange multiplier $\mu^{\varepsilon}$.
\begin{lemma}\label{le6}
There holds
\begin{equation*}
  \mu^\varepsilon \ge \frac{2\mathcal{E}_\varepsilon}{\kappa}-C.
\end{equation*}
Consequently,
\begin{equation*}
  \mu^\varepsilon\ge \frac{\kappa}{2\pi}\ln{\frac{1}{\varepsilon}}-C.
\end{equation*}
\end{lemma}

\begin{proof}
  We have
  \begin{equation*}
  \begin{split}
     2\mathcal{E}_\varepsilon(\zeta^\varepsilon) & =\int_D \zeta^\varepsilon\mathcal{G}\zeta^\varepsilon dx-2W\int_D x_1 \zeta^\varepsilon dx-2\mathcal{J}_\varepsilon \\
       & = \int_D \zeta^\varepsilon \psi^\varepsilon dx-W\int_D x_1 \zeta^\varepsilon dx-2\mathcal{J}_\varepsilon+\mu^\varepsilon \int_D \zeta^\varepsilon dx     \\
       & \le C+\kappa \mu^\varepsilon,
  \end{split}
  \end{equation*}
  which implies the desired result.
\end{proof}

By Lemmas \ref{le2}, \ref{le3} and \ref{le6}, we get
\begin{lemma}\label{kappa}
  If $\varepsilon$ is sufficiently small, then $\int_D \zeta^\varepsilon dx=\kappa$.
\end{lemma}

Now we turn to estimate the size of the supports of $\zeta^\varepsilon$ as $\varepsilon \to 0$. To
this end, we first recall an auxiliary lemma.
\begin{lemma}[\cite{Cao2}, Lemma 2.8]\label{le7}
  Let $\Omega\subset D$, $0<\varepsilon<1$, $\eta\ge0$, and let non-negative $\xi\in L^1(D)$, $\int_D \xi(x)dx=1$ and $||\xi||_{L^p(D)}\le C_1 \varepsilon^{-2(1-{1}/{p})}$ for some $1<p\le +\infty$ and $C_1>0$. Suppose for any $x\in \Omega$, there holds
  \begin{equation}\label{2-8}
    (1-\eta)\log\frac{1}{\varepsilon}\le \int_D \log\frac{1}{|x-y|}\xi(y)dy+C_2,
  \end{equation}
  where $C_2$ is a positive constant.
Then there exists some constant $R>1$ such that
\begin{equation*}
  diam(\Omega)\le R\varepsilon^{1-2\eta}.
\end{equation*}
The constant $R$ may depend on $C_1$, $C_2$, but not on $\eta$, $\varepsilon$.
\end{lemma}

Note that
\begin{equation*}
  G(x,y)\le \frac{1}{2\pi}\log\frac{1}{|x-y|}+C\ \ \  \text{in}\ D\times D,
\end{equation*}
and
\begin{equation*}
  \text{supp}(\zeta^\varepsilon)\subseteq \{x\in D\mid \mathcal{G}\zeta^\varepsilon(x)-Wx_1\ge \mu^\varepsilon\}.
\end{equation*}

Using Lemmas \ref{le6} and \ref{le7}, we immediately get the following estimate.
\begin{lemma}\label{le8}
Let $\zeta^{\varepsilon}$ be a maximizer as in Lemma \ref{le1}, then for $\varepsilon$ small it holds $\text{supp}(\zeta^\varepsilon)\le C\varepsilon$.
\end{lemma}

We have shown that the vorticies $\{\zeta^\varepsilon\}$ would shrink to some point in $\overline{D}$ when $\varepsilon \to 0$. Now we investigate this limiting location. Let
\begin{equation*}
  \mathcal{W}(t)=\frac{\kappa^2}{4\pi}\log\frac{1}{2t}+{\kappa Wt},\ \ t>0.
\end{equation*}
It is easy to see that $\mathcal{W}(r)=\min_{t>0} \mathcal{W}(t)$.

The following lemma show that the limiting location is $b_1$.
\begin{lemma}\label{le9}
  One has
  \begin{equation*}
    dist\left(supp(\zeta^\varepsilon), b_1\right)\to 0, \ \ \text{as}\ \varepsilon \to 0
  \end{equation*}
\end{lemma}
\begin{proof}
  Take $(x^\varepsilon_1,0)\in \text{supp}(\zeta^\varepsilon)$. Suppose (up to a subsequence) $x_1^\varepsilon\to \bar{r}$ as $\varepsilon \to 0$. We now show that $\bar{r}=r$. Set
  $\tilde{\zeta}^\varepsilon=\zeta^\varepsilon(\cdot-\bar{r}\mathbf{e}_1+r\mathbf{e}_1)\in \mathcal{A}$. Then
  \begin{equation*}
    \int_D \int_D{\tilde{\zeta}^\varepsilon(x)\log\frac{1}{|x-x'|}\tilde{\zeta}^\varepsilon}(x')dxdx'= \int_D \int_D{{\zeta}^\varepsilon(x)\log\frac{1}{|x-x'|}{\zeta}^\varepsilon}(x')dxdx',
  \end{equation*}
  and
  \begin{equation*}
 \int_D J(\varepsilon^2\tilde{\zeta}^\varepsilon)dx=\int_D J(\varepsilon^2{\zeta}^\varepsilon)dx.
  \end{equation*}
  Since $\mathcal{E}_\varepsilon(\tilde{\zeta}^\varepsilon)\le \mathcal{E}_\varepsilon(\zeta^\varepsilon)$, it follows that
  \begin{equation*}
  \begin{split}
    \frac{1}{4\pi} \int_D \int_D{{\zeta}^\varepsilon(x)\log\frac{1}{|\bar{x}-x'|}{\zeta}^\varepsilon}(x')dxdx'+W\int_{D}x_1{\zeta}^\varepsilon dx  &  \\
        \le \frac{1}{4\pi} \int_D \int_D{\tilde{\zeta}^\varepsilon(x)\log\frac{1}{|\bar{x}-x'|}\tilde{\zeta}^\varepsilon}(x')dxdx'&+W\int_{D}x_1\tilde{\zeta}^\varepsilon dx
  \end{split}
  \end{equation*}
  Letting $\varepsilon \to 0$, we get
  \begin{equation*}
    \mathcal{W}(\bar{r})\le \mathcal{W}(r).
  \end{equation*}
 This implies $\bar{r}=r$ and the proof is thus complete.
\end{proof}

Combining Lemma \ref{le8} and Lemma \ref{le9}, we get
\begin{lemma}\label{le10}
  For all sufficiently small $\varepsilon$, it holds $dist\left(supp(\zeta^\varepsilon), \partial D\right)>0$.
\end{lemma}

Now we show that $\psi^\varepsilon$ is a solution of \eqref{fr2} (with $\phi^\varepsilon=\mathcal{G\zeta^\varepsilon}$) when $\varepsilon$ is small.
\begin{lemma}\label{le11}
For all sufficiently small $\varepsilon$, one has
\begin{equation*}
  \zeta^{\varepsilon}=\frac{1}{ \varepsilon^2}f\left(\psi^{\varepsilon}\right) \ \ a.e.\  \text{in}\  \Pi.
\end{equation*}
\end{lemma}

\begin{proof}
By Lemma \ref{le2}, it suffices to show that $\psi^{\varepsilon}=0$ a.e. on $\Pi\backslash D$. By Lemma \ref{le10}, we have
\begin{equation*}
\begin{cases}
      -\Delta\psi^{\varepsilon}=0\ \ &\text{in}\  \Pi\backslash D, \\
     ~\psi^{\varepsilon}\le 0\ \ &\text{on}\  \partial \left( \Pi\backslash D\right), \\
     ~\psi^{\varepsilon} \le 0 \ \ &\text{at}\  \infty .
\end{cases}
\end{equation*}
By the maximum principle, we conclude that $\psi^{\varepsilon}\le0$ a.e. on $\Pi\backslash D$. The proof is therefore complete.
\end{proof}

\begin{lemma}\label{lem13}
The following asymptotic expansions hold as $\varepsilon\to 0^+$:
\begin{align}
\label{2-17}  \mathcal{E}_\varepsilon(\zeta^\varepsilon) & =\frac{\kappa^2}{4\pi}\ln\frac{1}{\varepsilon}+O(1), \\
\label{2-18}  \mu^\varepsilon & =\frac{\kappa}{2\pi}\ln\frac{1}{\varepsilon}+O(1).
\end{align}
\end{lemma}

\begin{proof}
  Using Riesz's rearrangement inequality, we obtain
  \begin{equation*}
     \mathcal{E}_\varepsilon(\zeta^\varepsilon)\le \frac{\kappa^2}{4\pi}\ln\frac{1}{\varepsilon}+C.
  \end{equation*}
  Combining this and Lemma \ref{le3}, we get \eqref{2-17}. Using the same argument as in the proof of Lemma \ref{le6}, we can easily get \eqref{2-18}. The proof is complete.
\end{proof}

\begin{remark}
  With the above results in hand, one can further study the shape of the free boundary $\{x\in \Pi\mid \psi^\varepsilon(x)=0\}$ by the standard scaling techniques (see \cite{Cao2, Tur83}). The vortex core will be approximately a disk.
\end{remark}

We are now in a position to prove Theorem \ref{thm1}.
\begin{proof}[Proof of Theorem \ref{thm1}]
It follows from the above lemmas by letting $\lambda=1/\varepsilon^{2}$.
\end{proof}

\subsection{Proof of Theorem \ref{thm2}}
In this subsection, we consider the case that the nonlinearity is unbounded. This situation seems a little more complicated. The key idea is to truncate the function $f$. Let $\rho>1$ and
\begin{equation*}
  f_\rho(s)=f(s)\ \ \text{if}\ s\le \rho;\ \ f_\rho(s)=f(\rho),\ \ \text{if}\ s>\rho.
\end{equation*}
Let
\begin{equation*}
  F_\rho(s)=\int_{0}^{s}f_\rho(t)dt
\end{equation*}
and $J_\rho$ be the conjugate function to $F_\rho$.
\subsubsection{Variational problem}
Let $D=\{x\in \Pi~|~r/2<x_1<2r,\ -1<x_2<1\}$ and
\begin{equation*}
\mathcal{A}=\{\zeta\in L^\infty(\Pi)~|~  \zeta \ge 0~ \text{a.e.}, \int_{\Pi}\zeta dx \le \kappa,~ supp(\zeta)\subseteq D \}.
\end{equation*}
For $\varepsilon>0$ we define $\mathcal{E}_{\varepsilon,\rho}(\zeta)\in [-\infty,+\infty)$ by
\begin{equation*}
  \mathcal{E}_{\varepsilon,\rho}(\zeta)=\frac{1}{2}\int_D{\zeta \mathcal{G}\zeta}dx-\frac{{W}}{2}\int_{D}x_1\zeta dx-\frac{1}{\varepsilon^2}\int_D J_\rho(\varepsilon^2\zeta)dx.
\end{equation*}

Using the same argument as in the proof of Theorem \ref{thm1}, we can obtain a family of solutions $(\zeta^{\varepsilon,\rho},\psi^{\varepsilon,\rho})$ which will shrink to $b_1$ when $\varepsilon$ tends to zero. Unfortunately, $\psi^{\varepsilon, \rho}$ is just a solution of the truncated equation, namely,
\begin{equation}\label{3-1}
   -\Delta \psi= \frac{1}{\varepsilon^2} f_\rho(\psi)\ \ \ \ \text{in}\ \ \Pi.
\end{equation}
If we can select a suitable $\rho$ such that $\psi^{\varepsilon,\rho}\le \rho$ almost everywhere in $\Pi$ when $\varepsilon$ is small, then the truncation of $f$ in \eqref{3-1} can be removed. In other words, the proof of Theorem \ref{thm2} is thus completed. We will demonstrate this in the next subsection.

 \subsubsection{Asymptotic behavior}
Recall that
\begin{equation*}
 \psi^{\varepsilon,\rho}(x)=\mathcal{G}\zeta^{\varepsilon,\rho}(x)-Wx_1-\mu^{\varepsilon,\rho},
\end{equation*}
and
\begin{equation*}
   \mathcal{G}\zeta^{\varepsilon,\rho}(x)=\int_\Pi G(x,y)\zeta^{\varepsilon,\rho}(y)dy=\frac{1}{2\pi} \int_D \log\frac{|\bar{x}-y|}{|x-y|}\zeta^{\varepsilon,\rho}(y)dy.
\end{equation*}
\begin{lemma}\label{le12}
There exists a constant $C>0$ not depending on $\varepsilon$ and $\rho$, such that
\begin{equation*}
   \mathcal{G}\zeta^{\varepsilon,\rho}(x)\le \frac{\kappa}{2\pi}\log\frac{1}{\varepsilon}+\frac{\kappa}{4\pi}\log f(\rho)+C,\ \ \forall~ x\in\Pi,
\end{equation*}
provided $\varepsilon$ is small enough.
\end{lemma}
\begin{proof}
  Since $\zeta^{\varepsilon,\rho}=\varepsilon^{-2}f_\rho(\psi^{\varepsilon,\rho})$, it follows that $0\le \zeta^{\varepsilon,\rho}\le f(\rho)/\varepsilon^2$ almost everywhere in $\Pi$. For $x\in D$, we have
  \begin{equation*}
  \begin{split}
     \mathcal{G}\zeta^{\varepsilon,\rho}(x) & = \frac{1}{2\pi} \int_D \log\frac{|\bar{x}-y|}{|x-y|}\zeta^{\varepsilon,\rho}(y)dy \\
       & \le \frac{1}{2\pi} \int_D \log\frac{1}{|x-y|}\zeta^{\varepsilon,\rho}(y)dy +C      \\
       & \le  \frac{1}{2\pi} \frac{f(\rho)}{\varepsilon^2}\int_{B_{\varepsilon\sqrt{\kappa/f(\rho)\pi}}(x)} \log\frac{1}{|x-y|}\zeta^{\varepsilon,\rho}(y)dy +C      \\
       & \le \frac{\kappa}{2\pi}\log\frac{1}{\varepsilon}+\frac{\kappa}{4\pi}\log f(\rho)+C,
  \end{split}
  \end{equation*}
  where the number $C>0$ does not depend on $\varepsilon$ and $\rho$ and we have used a simple rearrangement inequality. On the other hand, if $\varepsilon$ is sufficiently small, then $\psi^{\varepsilon,\rho}(x)\le 0$ on $\Pi \backslash D$. Hence the assertion is proved.
\end{proof}

By selecting a suitable competitor, we can easily get a lower bound of the energy.
\begin{lemma}\label{le13}
   There exists a constant $C>0$ not depending on $\varepsilon$ and $\rho$, such that
\begin{equation*}
  \mathcal{E}_{\varepsilon,\rho}(\zeta^{\varepsilon,\rho})\ge \frac{\kappa}{4\pi}\log{\frac{1}{\varepsilon}}-C.
\end{equation*}
\end{lemma}

The following result is an improvement of Lemma \ref{le6}.
\begin{lemma}\label{le14}
  For all sufficiently small $\varepsilon$, there holds
  \begin{equation*}
    \mu^{\varepsilon,\rho}\ge \frac{\kappa}{2\pi}\log\frac{1}{\varepsilon}-|1-2\vartheta_0|\rho-C,
  \end{equation*}
  where the constant $C>0$ does not depend on $\varepsilon$ and $\rho$.
\end{lemma}

\begin{proof}
  By assumption (H2), there exist $\vartheta_0\in(0,1)$ and $\vartheta_1>0$ such that
  \begin{equation}\label{3-3}
    [F(s)\leq \vartheta_0f(s)s+\vartheta_1f(s),\,\,\,\forall\,s\geq0.
  \end{equation}
Observe that
\begin{equation}\label{3-4}
  F_\rho(s)=F(s)\ \ \text{if}\ \ s\le \rho\,;\ \ F_\rho(s)=F(\rho)+f(\rho)(s-\rho)\ \ \text{if}\ \ s>\rho.
\end{equation}
By convexity, it holds
\begin{equation}\label{3-5}
   J_\rho(\varepsilon^2\zeta^{\varepsilon,\rho})+ F_\rho(\psi^{\varepsilon,\rho}) = \varepsilon^2 \zeta^{\varepsilon,\rho} \psi^{\varepsilon,\rho}, \ \ \text{a.e.}\ \text{on}\ D,
\end{equation}
Using \eqref{3-3}, \eqref{3-4} and \eqref{3-5}, we get
    \begin{equation*}
  \begin{split}
     2\mathcal{E}_{\varepsilon,\rho}(\zeta^{\varepsilon,\rho}) & =\int_D \zeta^{\varepsilon,\rho}\mathcal{G}\zeta^{\varepsilon,\rho} dx-2W\int_D x_1 \zeta^{\varepsilon,\rho} dx-\frac{2}{\varepsilon^2}\int_D J_\rho(\varepsilon^2\zeta^{\varepsilon,\rho})dx\\
       & \le -\int_D \zeta^{\varepsilon,\rho} \psi^{\varepsilon,\rho}dx+\frac{2}{\varepsilon^2}\int_D F_\rho(\zeta^{\varepsilon,\rho})+\kappa+3r\kappa W \mu^{\varepsilon,\rho}  \\
       & \le \frac{f(\rho)}{\varepsilon^2}\int_D (\psi^{\varepsilon,\rho}-\rho)_+dx +|1-2\vartheta_1|\rho+\kappa \mu^{\varepsilon,\rho}+3r\kappa W +2\vartheta_1\kappa,
  \end{split}
  \end{equation*}
  By Lemma \ref{le13}, we conclude that
  \begin{equation*}
     \mu^{\varepsilon,\rho}\ge \frac{\kappa}{2\pi}\log\frac{1}{\varepsilon}-|1-2\vartheta_0|\rho-\frac{f(\rho)}{\varepsilon^2}\int_D (\psi^{\varepsilon,\rho}-\rho)_+dx-(3rW+2\vartheta_1).
  \end{equation*}
  So it remains to show that $f(\rho)\varepsilon^{-2}\int_D (\psi^{\varepsilon,\rho}-\rho)_+dx$ is uniformly bounded with respect to $\varepsilon$ and $\rho$. Taking $U:=(\psi^{\varepsilon,\rho}-\rho)_+$ as a test function, by H\"older's inequality and Sobolev's inequality we have
  \begin{equation}\label{3-6}
  \begin{split}
     \int_D |\nabla U|^2 dx & =\int_D \zeta^{\varepsilon,\rho} U dx \\
       & \le f(\rho)\varepsilon^{-2} |\text{supp}(U)|^\frac{1}{2}\left(\int_D U^2 dx\right)^\frac{1}{2}  \\
       & \le Cf(\rho)\varepsilon^{-2} |\text{supp}(U)|^\frac{1}{2}\left(\int_D |\nabla U| dx+\int_D U dx  \right)\\
       & \le C \left(\int_D |\nabla U|^2 dx \right)^\frac{1}{2}+C\varepsilon \int_D \zeta^{\varepsilon,\rho} U dx,
  \end{split}
  \end{equation}
  where $C>0$ is independent of $\varepsilon$ and $\rho$. From \eqref{3-6}, we conclude that $f(\rho)\varepsilon^{-2}\int_D (\psi^{\varepsilon,\rho}-\rho)_+dx$ is uniformly bounded with respect to $\varepsilon$ and $\rho$. The proof is thus completed.
 \end{proof}

Combining Lemma \ref{le12} and Lemma \ref{le14}, we get a priori upper bound of $\psi^{\varepsilon, \rho}$ with respect to $\rho$.
\begin{lemma}
There exists a constant $C_0>0$ not depending on $\varepsilon$ and $\rho$, such that
\begin{equation*}
  \psi^{\varepsilon,\rho}(x)\le \frac{\kappa}{4\pi}\log f(\rho)+|1-2\vartheta_0|\rho+C_0,\ \ \forall\, x\in \Pi,
\end{equation*}
when $\varepsilon$ is sufficiently small.
\end{lemma}

Now we are ready to prove Theorem \ref{thm2}.  By assumption (H3), we can find a number $\rho_0>0$ such that
\begin{equation*}
  \frac{\kappa}{4\pi}\log f(\rho_0)+|1-2\vartheta_0|\rho_0+C_0 \le \rho_0.
\end{equation*}
This means that $\psi^{\varepsilon,\rho_0}\le \rho_0$ almost everywhere in $\Pi$ when $\varepsilon$ is small enough. Hence
\begin{equation*}
   -\Delta \psi^{\varepsilon,\rho_0}= \frac{1}{\varepsilon^2} f(\psi^{\varepsilon,\rho_0})\ \ \ \ \text{in}\ \ \Pi.
\end{equation*}
The proof of Theorem \ref{thm2} is thus completed.

{\bf Acknowledgments.}
{ This work was supported by NNSF of China Grant 11831009 and Chinese Academy of Sciences (No.
QYZDJ-SSW-SYS021).
}

\phantom{s}
 \thispagestyle{empty}

\end{document}